\documentclass{amsart}
\usepackage{amsfonts}
\usepackage{amsmath}
\usepackage{amssymb}
\usepackage{amscd}
\usepackage{MnSymbol}

\usepackage[all,cmtip]{xy}

\newtheorem{thm}{Theorem}[section]
\newtheorem{lm}[thm]{Lemma}
\newtheorem{cor}[thm]{Corollary}
\newtheorem{prop}[thm]{Proposition}

\theoremstyle{remark}

\theoremstyle{remark}
\newtheorem{exm}[thm]{Example}

\DeclareMathOperator{\Ker}{\rm ker}
\DeclareMathOperator{\End}{\rm End}

\newcommand{\N}{\mathcal{N}}

\newcommand{\Ac}{\mathcal{A}}
\newcommand{\Kc}{\mathcal{K}}

\newcommand{\Dc}{\mathcal{D}}

\newcommand{\Nc}{\mathcal{N}}

\newcommand{\Fc}{\mathcal{F}}
\newcommand{\Pc}{\mathcal{P}}
\newcommand{\Ic}{\mathcal{I}}
\newcommand{\Gc}{\mathcal{G}}

\newcommand{\Z}{\mathbb{Z}}
\newcommand{\Q}{\mathbb{Q}}

\newcommand{\Cc}{\mathcal{C}}

\newcommand{\M}{\mathcal{M}}

\newcommand{\Sc}{\mathcal{S}}
\newcommand{\Tc}{\mathcal{T}}

\newcommand{\onu}{\overline{\nu}}

\newcommand{\omu}{\overline{\mu}}

\newcommand{\opi}{\overline{\pi}}
\newcommand{\orho}{\overline{\rho}}

\newcommand{\Hom}{\mbox{Hom}}

\DeclareMathOperator{\Add}{{\rm Add}_\mathcal{A}}

\subjclass[2000]{16D10 (16S50)}
\keywords{additive category, Ab5 category, autocompact object}
\thanks{This work is part of the project SVV-2020-260589}

\begin{document}
\title{Autocompact objects of Ab5 categories}
\author{Josef Dvo\v r\'ak}
\address{CTU in Prague, FEE, Department of mathematics, Technick\'a 2, 166 27 Prague 6 \&
MFF UK, Department of Algebra,  Sokolovsk\' a 83, 186 75 Praha 8, Czech Republic} 
\email{pepa.dvorak@post.cz}

\author{Jan \v Zemli\v cka}
\address{Department of Algebra, Charles University,
	Faculty of Mathematics and Physics Sokolovsk\' a 83, 186 75 Praha 8, Czech Republic} 
\email{zemlicka@karlin.mff.cuni.cz}

\begin{abstract}
The aim of the paper is to describe autocompact objects in Ab5-categories, i.e. objects in cocomplete abelian categories with exactness preserving filtered colimits of exact sequences, whose covariant Hom-functor commutes with copowers of the object itself. A characterization of non-autocompact object is given, a general criterion of autocompactness of an object via the structure of its endomorphism ring is presented and a criterion of autocompactness of products is proven.
\end{abstract}
\date{\today}
\maketitle

\section*{Introduction}

An object $C$ of an abelian category $\Ac$ closed under coproducts is said to be {\it autocompact}, if the corresponding covariant hom-functor $\Ac(C,-)$ with target category being the category of abelian groups commutes with coproducts $C^{(\kappa)}$ for all cardinals $\kappa$, i.e. there is a canonical abelian group isomorphism between objects $\Ac(C,C^{(\kappa)})$ and $\Ac(C,C)^{(\kappa)}$. Note that it generalizes the profoundly treated notion of compact objects whose covariant hom-functors commute with arbitrary coproducts.

A systematic study of compact objects in categories of modules began in late 60's with Hyman Bass remarking in \cite[p.54]{Bass68} that the class of compact modules extends the class of finitely generated ones. This observation was elaborated in the work of Rudolf Rentschler \cite{Ren69}, where he presented basic constructions and conditions of existence of infinitely generated compact modules. The attention to autocompact objects within the category of abelian groups was then attracted by the work \cite{A-M}. The later research was motivated mainly by progress in the structural theory of abelian groups  \cite{ABW,Br03,BS12} and modules \cite{AB,BZ,Mo19}.
Although the notions of compactness and autocompactness were in fact studied in various algebraic contexts and with heterogeneous motivation 
(structure of modules \cite{Head,EklGooTrl97}, graded rings \cite{GomMilNas94}, representable equivalences of module categories \cite{ColMen93}, the structure of almost free modules \cite{Tr95}),
their overall categorial nature was omitted for a long time. Nevertheless, there have been several recent papers dedicated to the description of compactness in both non-abelian \cite{Mo10,DZ20} and abelian \cite{KZ19} categories published. 

The present paper follows the undertaking begun with \cite{KZ19} and its main goal is not only to survey results concerning self-small abelian groups and modules from the standpoint of abelian categories, but it tries to deepen and extend some of them in a way that they could be applied back in the algebraic context. We initiate with an investigation of the more general concept of relative compactness.
The second section summarizes some basic tools developed in \cite{KZ19}, which allows for the description of structure and closure properties of relative compactness, in particular, Proposition~\ref{sum-compact} shows that $\bigoplus\M$ is $\bigoplus\Nc$-compact for finite families of objects $\M$ and $\Nc$ of an Ab5-category if and only if $M$ is $N$-compact for all $M\in\M$ and $N\in\Nc$. The third section presents a general criterion of an object to be autocompact via the structure of its endomorphism ring (Theorem~\ref{notAC}) and, as a consequence, a description of autocompact coproducts (Proposition~\ref{finite_sums}). The main result of the paper presented in Theorem~\ref{criterion} which proves that $\prod\M$ is an autocompact object if and only if it is $\bigoplus\M$-compact.

\section{Preliminaries}

A category with a zero object is called \emph{abelian} if the following four conditions are satisfied:
\begin{enumerate}
\item for each finite discrete diagram the product and coproduct exist and they are canonically isomorphic,
\item each $\Hom$-set has a structure of an abelian group such that the composition of morphisms is bilinear,
\item with each morphism it contains its kernel and cokernel, 
\item monomorphisms are kernels of suitable morphisms, while epimorphisms are cokernels of suitable morphisms. 
\end{enumerate}
A category is said to be \emph{complete} (\emph{cocomplete}) if it contains limits (colimits) of all small diagrams;
 a cocomplete abelian category where all filtered colimits of exact sequences preserve exactness is then called an \emph{Ab5} category. 

Any small discrete diagram is said to be a \emph{family}.
Let $\M$ be a family of objects from $\Ac$; then the corresponding coproduct (product) is denoted $(\bigoplus\M,(\nu_M \mid M \in \M))$ ($(\prod\M,(\pi_M \mid M\in\M))$) and $\nu_M$ ($\pi_M$) are called \emph{structural morphisms} of the coproduct (of the product). In case $\M=\{M_i\mid i\in K\}$ with $M_i=M$ for all $i\in K$, where $M$ is an object of $\Ac$, we shall write $M^{(K)}$ ($M^{K}$) instead of $\bigoplus\M$ ($\prod\M$) and the corresponding structural morphisms shall be denoted by $\nu_i:=\nu_{M_i}$
($\pi_i:=\pi_{M_i}$ resp.) for each $i\in K$.

Let $\Nc$ be a subfamily of $\M$. Following the terminology set in \cite{KZ19} the coproduct $(\bigoplus\Nc,(\onu_N \mid N\in\Nc))$ in $\Ac$ is called a \emph{subcoproduct} and dually the product $(\prod\Nc,(\opi_N \mid N\in\Nc))$ is said to be a \emph{subproduct}. 
Recall there exists a unique canonical morphism $\nu_\Nc \in \Ac \left( \bigoplus\Nc, \bigoplus\M \right)$ ($\pi_\Nc \in \Ac \left(\prod\M, \prod\Nc\right)$) given by the universal property of $\bigoplus\Nc$ ($\prod\Nc$) satisfying $\nu_N=\nu_\Nc \circ \onu_N$ ($\pi_N=\opi_N \circ \pi_\Nc$) for each $N\in\Nc$, to which we shall refer as to the \emph{structural morphism} of the subcoproduct (the subproduct) over a subfamily $\Nc$ of $\M$. If $\M=\{M_i\mid i\in K\}$ and $\Nc=\{M_i\mid i\in L\}$ where $M_i=M$ for an object $M$ and for $i$ from index sets $L\subseteq K$, the corresponding structural morphisms are denoted by $\nu_L$
and $\pi_L$ respectively. The symbol $1_M$ denotes the identity morphism of an object $M$ and the phrase \emph{the universal property of a limit (colimit)} refers to the existence of unique morphism into the limit (from the colimit).

For basic properties of introduced notions and unspecified terminology we refer to \cite{Pop73,F}.

Throughout the whole paper we assume that $\Ac$ is an \emph{Ab5} category.

\section{$\Cc$-compact objects}

In order to capture in detail the idea of relative compactness, which is the central notion of this paper, 
let us suppose that $M$ is an object of  the category $\Ac$ and 
$\Nc$ is a family of objects of $\Ac $. Note that the functor 
$\Ac(M,-)$ on any additive category maps into $\Hom$-sets with a structure of abelian groups, which allows for a 
definition of the mapping
\begin{equation*}
\Psi_\Nc: \bigoplus\left( \Ac(M,N) \mid N\in\Nc\ \right) \to \Ac(M,\bigoplus\Nc)
\end{equation*}
by the rule
\begin{equation*}
\Psi_\Nc(\varphi)=\nu_\Fc \circ \nu^{-1} \circ \pi_\Fc \circ \tau
\end{equation*}
where for the element $\varphi=(\varphi_N \mid N\in\Nc)$ of the abelian group $\bigoplus \left( \Ac(M,N) \mid N\in\Nc \right)$ the symbol $\Fc$ denotes the finite family $\{N\in\Nc\mid \varphi_N\ne 0\} $, the morphism $\nu \in \Ac(\bigoplus \Fc, \prod\Fc)$ is the canonical isomorphism and  $\tau \in \Ac(M, \prod \Nc)$ is the unique morphism given by the universal property of the product $(\prod \Nc,(\pi_N \mid N\in\Nc))$ applied on the cone $(M,(\varphi_N \mid  N\in\Nc))$, i.e. $\pi_N \circ \tau=\varphi_N$ for each $N \in \Nc$:
\begin{equation*}
\xymatrix{ 
	M  \ar@{~>}[r]^{\tau}\ar@{->}[dr]_{\varphi_{N}}  & \prod\Nc \ar@{->}[r]^{\pi_{\Fc}} \ar@{->}[d]^{\pi_{N}} & \prod\Fc \ar@{->}[r]^{\nu^{-1}} & \bigoplus \Fc \ar@{->}[r]^{\nu_{\Fc}} & \bigoplus \Nc \\
	& N
}
\end{equation*}

Recall a key observation regarding the algebraic concept of compactness:

\begin{lm}\cite[Lemma 1.3]{KZ19}\label{2.1}
For each family of objects $\Nc \subseteq \Ac$, the mapping $\Psi_\Nc$ is a monomorphism in the category of abelian groups.
\end{lm}

Let $M$ be an object and $\Cc$ a class of objects of the category $\Ac$. In accordance with \cite{KZ19}, $M$ is called \emph{$\Cc$-compact} if $\Psi_\Nc$ is an isomorphism for each family $\Nc \subseteq\Cc$.
For objects $M,N \in Ac$ we say that $M$ is {\it $N$-compact} (or relatively compact over $N$)
if it is an $\{N\}$-compact object and $M$ is said to be \emph{autocompact}
whenever it is $M$-compact.

\begin{exm}\label{ex1} (1) If $M$ and $N$ are objects such that $\Ac(N,M)=0$, then $N$ is $M$-compact object, in particular $\mathbb Q$ is a $\mathbb Z$-compact abelian group. 

(2) Self-small right modules over a unital associative ring, in particular finitely generated ones, are autocompact objects in the category of all right modules. 
\end{exm}

Let us formulate an elementary but useful observation:

\begin{lm}\label{inclusion-comp}
Let $M$ be an object and let $\mathcal{B}\subseteq\Cc$ be families of objects of the category $\Ac$. If $M$ is $\Cc$-compact, then it is $\mathcal{B}$-compact.
\end{lm}

We shall need two basic structural observations concerning
the category $\Ac$ formulated in \cite{KZ19}, which express relationship between coproducts and products using their structural morphisms.
For the convenience of the reader we quote both of the results, the first one is formulated for the special case of products coproducts of copies of $M$, while the second one is kept in the original form.
 
\begin{lm}\cite[Lemma 1.1]{KZ19}\label{Morph} 
Let $M$ be an object of $\Ac$ and $L\subseteq K$ be sets. If $\Ac$ contains products $(M^L,(\overline{\pi}_i \mid i\in L))$ and $(M^K,(\pi_i \mid i\in K))$, then
\begin{enumerate}
\item There exist unique morphisms $\rho_L \in \Ac(M^{(K)}, M^{(L)})$ and $\mu_L \in \Ac(M^{L}, M^{K})$ such that 
$\rho_L \circ \nu_i=\onu_i$, $\pi_i \circ \mu_L=\opi_i$ for $i\in L$, and $\rho_L \circ \nu_i=0$, $\pi_i \circ \mu_L = 0$ for $i\notin L$.  
\item For each $i\in K$ there exist unique morphisms 
$\rho_i \in \Ac(M^{(K)}, M)$ and  $\mu_i \in \Ac(M, M^{K})$ such that $\rho_i \circ \nu_i=1_{M}$, $\pi_i \circ \mu_i=1_M$ and $\rho_i \circ \nu_j=0$, $\pi_j \circ \mu_i=0$ whenever $i\ne j$. Denoting by $\orho_i$ and $\omu_i$ the corresponding morphisms for $i \in L$, we have $\mu_L \circ \omu_j = \mu_j$ and $\rho_L \circ \orho_j = \rho_j$ for all $j \in L$.
\item There exists a unique morphism $t \in \Ac(M^{(K)}, M^{K})$
such that $\pi_i \circ t=\rho_i$ and $t \circ \nu_i = \mu_i$ for each $i\in K$.  
\end{enumerate}
\end{lm}

\begin{lm}\cite[Lemma 1.1(i) and 1.2]{KZ19}\label{StructMorph}
Let $\Nc\subseteq \M$ be families of objects of $\Ac$ and let there exist products $(\prod \Nc,(\overline{\pi}_N \mid N\in\Nc))$ and $(\prod \M,(\pi_N \mid N\in\M))$ in $\Ac$.
\begin{enumerate}
\item There exist unique morphisms $\rho_\Nc \in \Ac(\bigoplus\M, \bigoplus\Nc)$ and $\mu_\Nc \in \Ac(\prod\Nc, \prod\M)$ such that 
$\rho_\Nc \circ \nu_\Nc=1_{\bigoplus\Nc}$,
$\pi_\Nc \circ \mu_\Nc=1_{\prod\Nc}$
and $\rho_\Nc \circ \nu_M=0$, $\pi_M \circ \mu_\Nc = 0$  for each $M\notin\Nc$.
\item There exist unique morphisms $\overline{t} \in \Ac\left( \bigoplus \Nc, \prod \Nc \right)$ and $t \in \Ac\left( \bigoplus \M, \prod \M \right)$ 
such that $\pi_N \circ t=\rho_N$ and $t \circ \nu_N = \mu_N$ for each $N\in \M$, $\opi_N \circ \overline{t}=\orho_N$ and $\overline{t} \circ \onu_N = \omu_N$ for each $N\in \Nc$. Furthermore, the diagram
\begin{equation*}
\xymatrix{ 
	\bigoplus\N  \ar@{->}[r]^{\nu_\Nc} \ar@{->}[d]_{\overline{t}} & \bigoplus\M \ar@{->}[r]^{\rho_\Nc}\ar@{->}[d]^{t} & \bigoplus\N  \ar@{->}[d]^{\overline{t}} \\
	\prod\Nc \ar@{->}[r]^{\mu_\Nc}  & \prod \M\ar@{->}[r]^{\pi_\Nc}  & \prod\N
}
\end{equation*}
commutes.

\item Let $\kappa$ be an ordinal and let $(\Nc_{\alpha} \mid \alpha<\kappa)$ be a disjoint partition of $\M$. For $\alpha < \kappa$ set $S_\alpha:= \bigoplus \Nc_{\alpha}$, $P_{\alpha}:=\prod \Nc_{\alpha}$ and denote families of the corresponding limits and colimits as $\Sc: = (S_{\alpha} \mid \alpha < \kappa)$, $\Pc: = (P_{\alpha} \mid \alpha < \kappa)$. Then $\bigoplus\M \simeq \bigoplus{\Sc}$ and $\prod\M \simeq \prod{\Pc}$ where both isomorphisms are canonical, i.e. for each object $M\in\M$ the following diagrams commute:
\begin{equation*}
\xymatrix{
	M  \ar@{->}[r]^{\nu^{(\alpha)}_M} \ar@{->}[d]_{\nu_M} & S_{\alpha} \ar@{->}[d]^{\nu_{S_{\alpha}}} \\
	\bigoplus \M  \ar@{~>}[r]^{\simeq }  & \bigoplus \Sc  } 
	\ 
	\xymatrix{
	\prod \Pc \ar@{~>}[r]^{\simeq} \ar@{->}[d]_{\pi_{P_{\alpha}}} & \prod \M \ar@{->}[d]^{\pi_M} \\
	P_{\alpha} \ar@{->}[r]^{\pi^{(\alpha)}_M} & M}
\end{equation*}
\end{enumerate}
\end{lm}

Morphisms $\rho_{L}$, $\rho_\Nc$,  ($\mu_{L}$, $\mu_\Nc$) from Lemma~\ref{Morph}(1) and Lemma~\ref{StructMorph}(1) are called the \emph{associated morphisms} to the structural morphisms $\nu_{L}$, $\nu_\Nc$ ($\pi_{L}$, $\pi_\Nc$) over the subcoproduct (the subproduct) of $M$. The unique morphism $t$ from Lemma~\ref{StructMorph}(2) is said to be the \emph{compatible coproduct-to-product} morphism. Note that in an Ab5-category $t$ is a monomorphism by \cite[Chapter 2, Corollary 8.10]{Pop73} and if $K$ is finite, it is by definition an isomorphism.

We translate now a general criteria \cite[Lemma 1.4, Theorem 1.5]{KZ19} of categorial $\Cc$-compactness to the description of $N$-compactness for an arbitrary object $N$:

\begin{thm}\label{N-comp}  
The following conditions are equivalent for objects $M$ and $N$ of the category $\Ac$:
\begin{enumerate}
\item $M$ is $N$-compact,
\item for each cardinal $\kappa$ and $f \in \Ac(M,  N^{(\kappa)})$ there exists a finite set $F \subset \kappa$
and a morphism $f' \in \Ac(M, N^{(F)})$ such that  $f=\nu_F \circ f'$.
\item for each cardinal $\kappa$ and $f \in \Ac(M,  N^{(\kappa)})$ there exists a finite set $F \subset \kappa$ such that  $f = \sum\limits_{\alpha\in F} \nu_\alpha \circ \rho_\alpha \circ f$,
\item for each morphism $\varphi \in \Ac(M,N^{(\omega)})$ there exists $\alpha<\omega$ such that $\rho_\alpha \circ \varphi= 0$.
\item there exists a family of $N$-compact objects $\Gc$ and
an epimorphism $e \in \Ac(\bigoplus\Gc, M)$ such that for each 
countable family $\Gc_\omega \subseteq \Gc$ there exists 
 a non-$N$-compact object $F$ and morphism $f\in\Ac(F,M)$
 such that $f^c \circ e \circ \nu_{\Gc_\omega}= 0$ for the cokernel $f^c$ of $f$.
\end{enumerate}
\end{thm}
\begin{proof} Equivalences 
$(1) \Leftrightarrow (2) \Leftrightarrow (3)$ follow immediately from 
\cite[Lemma 1.4]{KZ19}, while $(1) \Leftrightarrow (4)\Leftrightarrow (5)$ are consequences of 
\cite[Theorem 1.5]{KZ19}
\end{proof}

\section{Correspondences of compact objects}

As the base step of our research we describe $C$-compact objects for a single object $C$ of an Ab5 category $\Ac$. Let us begin with the observation that we can study $C$-compactness of a suitable object instead of the compactness over a set of objects.

Let us denote the class 
\[
\Add(\Cc)=\{A\mid \exists B, \exists \kappa, \forall \alpha<\kappa, \exists C_\alpha\in\Cc: A\oplus B\cong \bigoplus_{\alpha<\kappa}C_\alpha \}
\]
for every family $\Cc$ of objects of $\Ac$ and put $\Add(C):=\Add(\{C\})$.

\begin{lm}\label{Autocompact3} 
The following conditions are equivalent for an object $M$ and a set of objects $\Cc$ of the category $\Ac$:
\begin{enumerate}
\item $M$ is $\bigoplus\Cc$-compact,
\item $M$ is $\Cc$-compact,
\item $M$ is $\Add(\bigoplus\Cc)$-compact,
\item $M$ is $\Add(\Cc)$-compact.
\end{enumerate}
\end{lm}
\begin{proof} Since $\Add(\bigoplus\Cc)=\Add(\Cc)$, the equivalence
(3)$\Leftrightarrow$(4) is obvious. Implications 
(3)$\Rightarrow$(1) and (4)$\Rightarrow$(2) are clear from Lemma~\ref{inclusion-comp}.

(2)$\Rightarrow$(4) Let $\varphi\in \Ac(M,\bigoplus\mathcal D)$ for a family $\Dc$ of objects of $\Add(\bigoplus\Cc)$. For each $D\in \mathcal D$ there exists a family $\Cc_D$ of objects of $\Cc$ and a monomorphism $\nu_D: D\to \bigoplus\Cc_D$, hence there exists a monomorphism 
$\nu: \bigoplus\mathcal D\to \bigoplus_{D\in \mathcal D}\bigoplus\Cc_D$.
Since $M$ is $\Cc$-compact, the morphism $\nu\varphi$ factorizes through a finite subcoproduct by \cite[Lemma 1.4]{KZ19}, hence $\varphi$
 factorizes through a finite subcoproduct, so $M$ is $\Add(\Cc)$-compact by
\cite[Lemma 1.4]{KZ19} again.

(1)$\Rightarrow$(3) Follows from the implication (2)$\Rightarrow$(4) where we take $\{\bigoplus\Cc\}$ instead of $\Cc$.
\end{proof}

\begin{cor}\label{sum-compact0}   If $\Nc\subseteq \M$ are families of objects such that $\Nc$ contains infinitely many nonzero objects, then $\bigoplus\M$ is not $\Nc$-compact, so it is not $\bigoplus\Nc$-compact.
\end{cor}

Since $\Add(M)=\Add(M^{(n)})$ for any integer $n$, we have the following consequence:

\begin{cor}\label{M^k} Let $\kappa$ be a cardinal and $M$ an autocompact object.
Then $M^{(\kappa)}$ is autocompact if and only if $\kappa$ is finite. 
\end{cor}

The next result shows the correspondence between classes of compact objects over different pairs of objects.

\begin{lm}\label{Autocompact1}  
Let $A,B,M$ be objects of $\Ac$ and let there exist a cardinal $\lambda$ and a monomorphism $\mu\in\Ac(A, B^\lambda)$. 
If $M$ is $B$-compact, then $M$ is $A$-compact.
\end{lm}
\begin{proof} Denote by $\nu_{\alpha}$ and $\tilde{\nu}_{\alpha}$ the corresponding structural morphisms of coproducts $A^{(\omega)}$ and $B^{(\omega)}$, and by $\rho_{\alpha}$ and $\tilde{\rho}_{\alpha}$ their associated morphisms, respectively.

Suppose that $M$ is not $A$-compact. Then there exists $\varphi \in \Ac(M,A^{(\omega)})$
such that $\rho_\alpha \varphi\ne 0$ for all $\alpha<\omega$ by Theorem~\ref{N-comp}.
Since $\mu$ is a monomorphism by assumption, we get that $\mu\rho_\alpha \varphi\ne 0$, which implies that there 
exists $\beta_\alpha<\lambda$ such that 
$\pi_{\beta_\alpha}\mu\rho_\alpha  \varphi\ne 0$ for each $\alpha<\omega$ by the universal property of the product $B^\lambda$.
Put $\mu_\alpha=\pi_{\beta_\alpha}\mu\in \Ac(A,B)$ and note  we have proved that $\mu_\alpha\rho_\alpha  \varphi$ is a nonzero morphism $M\to B$ for each $\alpha<\omega$.

The  universal property of the coproduct $A^{(\omega)}$ implies that there exists a uniquely determined morphism $\psi\in \Ac(A^{(\omega)},B^{(\omega)})$ for which the diagram
\begin{equation*}
\xymatrix{
	A  \ar@{->}[r]^{\mu_\alpha} \ar@{->}[d]_{\nu_\alpha} & B \ar@{->}[d]^{\tilde{\nu}_\alpha} \\
	A^{(\omega)}  \ar@{~>}[r]^{\psi }  & B^{(\omega)} \ar@{->}[r]^{\tilde{\rho}_\gamma}  & B } 
\end{equation*}
commutes, i.e. we have equalities
$\psi\nu_\alpha = \tilde{\nu}_\alpha\mu_\alpha$ and
$ \tilde{\rho}_\gamma\psi\nu_\alpha = \tilde{\rho}_\gamma\tilde{\nu}_\alpha\mu_\alpha$
for each $\alpha,\gamma<\omega$. Hence for every $\alpha<\omega$ we get
$\tilde{\rho}_\alpha\psi\nu_\alpha =\mu_\alpha$
and  
$ \tilde{\rho}_\gamma\psi\nu_\alpha = 0$ whenever $\gamma\ne\alpha$ by Lemma~\ref{Morph}(2). 
Note that it means that $\tilde{\rho}_\gamma\psi\nu_\alpha=\tilde{\rho}_\gamma\psi\nu_\gamma\rho_\gamma\nu_\alpha$
for all $\alpha,\gamma<\omega$

By applying Theorem~\ref{N-comp} again we need to show that $\tilde{\rho}_\gamma\psi\varphi\ne 0$ for all $\gamma<\omega$.
The universal property of the coproduct $A^{(\omega)}$ implies that for every $\gamma<\omega$ there exists a unique morphism $\tau_\gamma\in \Ac(A^{(\omega)},B)$ such that the diagram
\begin{equation*}
\xymatrix{
	A  \ar@{->}[r]^{\nu_\alpha} \ar@{->}[d]_{\nu_\alpha} & A^{(\omega)}  \ar@{->}[d]^{\tilde{\rho}_\gamma\psi} \\
	A^{(\omega)}  \ar@{~>}[r]^{\tau_\gamma}  & B } 
\end{equation*}
commutes for each $\alpha<\omega$.
Since $\tilde{\rho}_\gamma\psi\nu_\alpha=\tilde{\rho}_\gamma\psi\nu_\gamma\rho_\gamma\nu_\alpha$
for all $\alpha,\gamma<\omega$, we get the equality
$\tilde{\rho}_\gamma\psi=\tau_\gamma=\tilde{\rho}_\gamma\psi\nu_\gamma\rho_\gamma$
by the universal property of the coproduct $A^{(\omega)}$.
Now, it remains to compute for every $\gamma<\omega$
\[
\tilde{\rho}_\gamma\psi\varphi 
=\tau_\gamma\varphi=\tilde{\rho}_\gamma\psi\nu_\gamma\rho_\gamma\varphi
= \mu_\gamma\rho_\gamma\varphi \ne 0,
\]
so $M$ is not $B$-compact by Theorem~\ref{N-comp}.
\end{proof}

\begin{cor}\label{Autocompact2}  
Let $M$ and $N$ be objects such that there exists a cardinal $\lambda$ and a monomorphism $\mu\in\Ac(M, N^\lambda)$ and $M$ is $N$-compact, then
$M$ is autocompact. 
\end{cor}

As another consequence of Lemma~\ref{Autocompact1} we can observe that general compactness can be tested by a single object.
Recall that the object $E$ the category $\Ac$ is called \emph{cogenerator} if the functor $\Ac(-,E)$ is an embedding \cite[Section 3.3]{F}. It is well known that there is a monomorphism of $A\to E^{\Ac(A,E)}$ for any object $A$ and a cogenerator $E$. 

\begin{prop}\label{inj}  Let $M$ be an object and $E$ be a cogenerator  of $\Ac$ such that $\Ac$ contains the product $ E^{\Ac(M,E)}$. Then $M$ is $E$-compact if and only if it is compact.
\end{prop} 

\begin{proof} Clearly, it is enough to prove the direct implication.
Let $M$ be $E$-compact and $\M$ be a family of objects. Since $E$ is a cogenerator , there exists a cardinal $\lambda$ and a monomorphism $\mu\in\Ac(\bigoplus\M , E^\lambda)$.
Then $M$ is $\bigoplus\M$-compact by Lemma~\ref{Autocompact1} and so $\M$-compact by Lemma~\ref{Autocompact3}. 
\end{proof}	

The rest of this section is dedicated to description of relative compactness over finite coproducts of finite coproducts of objects.

\begin{lm}\label{sum}  Let $A$ be an object and $\M$ a finite family of objects.
\begin{enumerate}
\item If $N$ is $A$-compact for each $N\in\M$, then $\bigoplus\M$ is $A$-compact.
\item If $A$ is $N$-compact for each $N\in\M$, then $A$ is $\bigoplus\M$-compact.
\end{enumerate}
\end{lm}
\begin{proof}  (1)
Assume that  $\bigoplus\M$ is not $A$-compact. Then by Theorem~\ref{N-comp} there exists a morphism 
$\varphi\in \Ac(\bigoplus\M,A^{(\omega)})$ with $\rho_n\varphi\ne 0$
for all associated morphisms $\rho_n$ of $A^{(\omega)}$.

Note that for each $n<\omega$ there exists some $N\in\M$ such that $\rho_n\varphi\nu_N\ne 0$ by the universal property of the coproduct $\bigoplus\M$ , where $\nu_N$ are the corresponding structural morphisms of $\bigoplus\M$. Therefore there exists $N\in\Nc$ for which the set
\[
I=\{n<\omega\mid \rho_n\varphi\nu_N\ne 0\}
\]
is infinite and  the morphism $\tilde{\varphi}=\rho_{I}\varphi\nu_N$ ensured by Lemma~\ref{Morph} satisfies 
$
\rho_n\tilde{\varphi} = \rho_n\rho_{I}\varphi\nu_N \ne  0
$ for each $n\in I$.
Now, Theorem~\ref{N-comp} implies that $N$ is not $A$-compact.

(2) Put $M=\bigoplus\M$ and denote by $\rho_i$, $\tilde{\rho}_i$ and $\rho_N$ the corresponding associate morphisms of coproducts $M^{(\omega)}$ and $N^{(\omega)}$ for each $N\in \M$ and $i<\omega$.
Denote furthermore  by 
$\rho_{N^{(\omega)}}\in\Ac(M^{(\omega)}, N^{(\omega)})$ the morphism given by Lemma~\ref{StructMorph} which satisfies $\tilde{\rho}_i\rho_{N^{(\omega)}}=\rho_N\rho_i$ for each $N\in \M$ and $i<\omega$.
Assume that  $A$ is not $\bigoplus\M$-compact: there exists a morphism $\varphi\in \Ac(A,M^{(\omega)})$ such that $\rho_n\varphi\ne 0$  for infinitely many $n$ by Theorem~\ref{N-comp} and now using the same argument as in the proof of (1)  we can find $N\in\M$ such that the set
\[
J=\{i<\omega\mid \rho_N\rho_i\varphi\ne 0\}
\]
is infinite. Since $\tilde{\rho}_i\rho_{N^{(\omega)}}\varphi = \rho_N\rho_i\varphi\ne 0$ for every $i\in J$, the object $A$ is not $N$-compact.
\end{proof}

\begin{lm}\label{A-compact}  
Let $M$, $N$ and $A$ be objects of $\Ac$ and $n$ be a natural number. If there exists an epimorphism $M^{(n)}\to N$ and $M$ is $A$-compact, then $N$ is $A$-compact.
\end{lm}
\begin{proof} 
Assume that  $N$ is not $A$-compact. Then there exists a morphism 
$\varphi\in \Ac(N,A^{(\omega)})$ such that $\rho_\alpha\varphi\ne 0$
for all associated morphisms $\rho_\alpha$ of $A^{(\omega)}$ by 
Theorem~\ref{N-comp}. If $\mu\in \Ac(M^{(n)}, N)$ is an epimorphism, $\rho_i\varphi\mu\ne 0$ for each $i<\omega$, hence $M^{(n)}$ is not $A$-compact. Then $M$ is not $A$-compact by Lemma~\ref{sum}(1).
\end{proof}

We can summarize the obtained necessary condition of autocompactness.
	
\begin{prop}\label{Autocompact4}  
Let $M$ and $N$ be objects of $\Ac$ such that there exists an epimorphism $M^{(n)}\to N$ for an integer $n$ and a monomorphism $N\to M^{\lambda}$ for a cardinal $\lambda$. If $M$ is is  autocompact, then $N$ is autocompact as well.
\end{prop}
\begin{proof}  
$N$ is $M$-compact, as follows from Lemma~\ref{A-compact}. Hence it is $N$-compact, so autocompact by Corollary~\ref{Autocompact2}.
\end{proof}

The next consequence presents a categorial version of the classical fact that an endomorphic image of a self-small module is self-small.

\begin{cor} If $M$ is an autocompact object, such that there exist an epimorphism $\epsilon\in\Ac(M,N)$ and a monomorphism $\mu\in\Ac(N,M)$, then $N$ is autocompact.
\end{cor}

\begin{exm} Let $A$ be a self-small right modules over a ring, i.e. autocompact object in the category of right modules. Denote 
$\mathcal K=\{\Ker f\mid f\in\End(A)\}$ and let $\mathcal L\subset \mathcal K$. Then $A/\bigcap \mathcal L$ is a self-small module by Proposition~\ref{Autocompact4} since there exist  monomorphisms
$A/\bigcap \mathcal L\hookrightarrow \prod_{L\in\mathcal L} A/L\hookrightarrow \prod_{L\in\mathcal L} A$ 
\end{exm}

We conclude the section mentioning closure properties of relatively compact objects.

\begin{prop}\label{sum-compact}  
Let $\M$ and $\Nc$ be finite families of objects. Then $\bigoplus\M$ is $\bigoplus\Nc$-compact if and only if $M$ is $N$-compact for all $M\in\M$ and $N\in\Nc$.
\end{prop}
\begin{proof} $(\Rightarrow)$ Since the associate morphism $\rho_N\in\Ac(\bigoplus\M,M)$ is an epimorphism for each $M\in\M$ and $\bigoplus\M$ is $\bigoplus\Nc$-compact, each object $M$ 
is $\bigoplus\Nc$-compact by Lemma~\ref{A-compact}.
As $\nu_N\in\Ac(N,\bigoplus\Nc)$ is a monomorphism for each $N\in\Nc$, any object $M\in\M$ 
is $N$-compact by Lemma~\ref{Autocompact1}.

$(\Leftarrow)$ Lemma~\ref{sum}(1) implies that $\bigoplus\M$ is $N$-compact for each $N\in\Nc$
and then Lemma~\ref{sum}(2) implies that $\bigoplus\M$ is $\bigoplus\Nc$-compact.
\end{proof}

\section{Description of autocompact objects}

This section is dedicated mainly to the generalization of a classical autocompactness criteria \cite{A-M} to an Ab5 category $\Ac$.

Assume $M$ is an object such that the category $\Ac$ is closed under  products $M^\lambda$ for all $\lambda\le|\End_\Ac(M)|$ and take $I\subseteq \End_\Ac(M)=\Ac(M,M)$.
Then there exists a unique morphism $\tau_I\in \Ac(M,M^I)$ satisfying $\pi_\iota\tau_I=\iota$ for each $\iota\in I$ by the universal property of the product $M^I$. Let us denote by $\Kc(I)=(K_I,\nu_I)$ the kernel of the morphism $\tau_I$ and note that $\Kc(I)$ is defined uniquely up to isomorphism.

For an object $K$ of $\Ac$ consider a morphism $\nu\in\Ac(K,M)$. We will then set
\[
\Ic(K,\nu)=\{\iota\in \End_\Ac(M)\mid \iota \nu=0 \}.
\]
It is easy to see that $\Ic(K,\nu)$ forms a left ideal of the endomorphism ring $\End_\Ac(M)$.
We say that a left ideal $I$ of $\End_\Ac(M)$ is an \emph{annihilator ideal} if $\Ic(\Kc(I))=I$.

\begin{lm}\label{annihilator}  Let $\Ac$ be closed under products $M^\lambda$ for all $\lambda\le|\End_\Ac(M)|$.
Then $\Ic(K,\nu)$ is an annihilator ideal of $\End(M)$ for all $\nu\in\Ac(K,M)$.
\end{lm}
\begin{proof} Put $I=\Ic(K,\nu)$, $(K_I,\nu_I)=\Kc(I)$ and $\tilde{I}=\Ic\Kc(I)=\Ic(K_I,\nu_I)$. Furthermore, denote by 
$\tau_I\in \Ac(K,M^I)$ the morphism satisfying $\pi_\iota\tau_I=\iota$ for each $\iota\in I$, i.e. $(K_I,\nu_I)$ is the kernel of $\tau_I$.
Since $\tau_I\nu_I=0$, we can easily compute that $\iota\nu_I=\pi_\iota\tau_I\nu_I=0$ for every $\iota\in I$, which implies $I\subseteq \tilde{I}$. 

To prove the reverse inclusion $\tilde{I}\subseteq I$, let us note that by the universal property of the kernel $\nu_I$ there exists a unique morphism $\alpha\in\Ac(K,K_I)$ such that all squares in the diagram
\begin{equation*}
\xymatrix{ 
	K  \ar@{->}[r]^{\nu} \ar@{~>}[d]_{\alpha} & M \ar@{->}[r]^{\tau_I}\ar@{=}[d]^{} & M^I \ar@{->}[d]^{\pi_\iota} \\
	K_I \ar@{->}[r]^{\nu_I}  & M \ar@{->}[r]^{\iota}  & M
}
\end{equation*}
commute for each $\iota\in I$. Consider a morphism $\gamma\in \End(M)$ such that $\gamma\notin I$. Then 
$\gamma\nu_I\alpha=\gamma\nu\ne 0$ by the definition of the ideal $I$.
Hence $\gamma\nu_I\ne 0$ and so $\gamma\notin \tilde{I}$.
\end{proof}

Recall the concepts of exactness and inverse limits in Ab5 categories.

The diagram $A_0\stackrel{\alpha_1}{\longrightarrow}  A_1\stackrel{\alpha_2}{\longrightarrow}\dots\stackrel{\alpha_{n-1}}{\longrightarrow}A_{n-1}\stackrel{\alpha_n}{\longrightarrow}A_{n}$ is said to be an \emph{exact sequence} provided for each $i=1,\dots,n-1$ the equality $\alpha_{i+1}\alpha_i=0$ holds and there exist an object $K_i$ together with morphisms $\xi_i\in\Ac(A_i,K_i)$ and $\theta_i\in\Ac(K_i,A_i)$ such that $(K_i,\theta_i)$ is a kernel of $\alpha_{i+1}$,  $(K_i,\xi_i)$ is a cokernel of $\alpha_i$ and $\xi_i\theta_i=1_{K_i}$.
In particular, the diagram $0\to A \stackrel{\alpha}{\longrightarrow}  B \stackrel{\beta}{\longrightarrow} C\to 0$ is  a \emph{short exact sequence} provided $\alpha$ is a kernel of $\beta$ and $\beta$ is a cokernel of $\alpha$, hence $\alpha$ is a monomorphism and $\beta$ is an epimorphism. Recall that any monomorphism (epimorphism) can be expressed as the first (second)  morphism of some short exact sequence in an Ab5-category.

A diagram $\Dc=(\{M_i\}_{i<\omega}, \{\nu_{i,j}\}_{i<j<\omega})$ is called an \emph{$\omega$-spectrum of $M$}, if $\nu_{i,j}\in\Ac(M_i,M_j)$, $\nu_{j,k}\nu_{i,j}=\nu_{i,k}$ for each $i<j<k<\omega$, and there exist morphisms $\nu_i\in\Ac(M_i,M)$ for all $i<\omega$ such that $(M,\{\nu_i\}_{i<\omega})$ is a colimit of the diagram $\Dc$ (i.e. it is a direct limit of the spectrum $\Dc$).

\begin{lm}\label{sums}
Let $M$ be an object and $M^{(\omega)}$ be a coproduct with structural morphisms $\nu_i$ and associated morphisms $\rho_i$, $i<\omega$. Put 
\[
n=\{0,\dots,n-1\},\ \ [n,\omega)=\omega\setminus n =\{i<\omega\mid i\ge n\},
\]
let $M^{(n)}$ and  $M^{([n,\omega))}$ be subcoproducts of $M^{(\omega)}$. Denote by $\nu_{(n,m)}\in\Ac(M^{(n)}, M^{(m)})$,  $\nu_{<n}\in \Ac(M^{(n)}, M^{(\omega)})$ the structural morphisms and by $\rho_{(n,m)}\in\Ac(M^{([n,\omega))}, M^{([m,\omega))})$,  $\rho_{\ge n}\in \Ac( M^{(\omega)}, M^{([n,\omega))})$ the associated morphisms given by Lemma~\ref{StructMorph} for all $n<m<\omega$. Then
\begin{enumerate}
\item for each $n<m<\omega$ all squares in the diagram with exact rows
\begin{equation*}
\xymatrix{ 
\mathfrak{M}_n: & 0 \ar@{->}[r] &	M^{(n)}  \ar@{->}[r]^{\nu_{<n}} \ar@{->}[d]_{\nu_{(n,m)}} & M^{(\omega)} \ar@{->}[r]^{\rho_{\ge n}}\ar@{=}[d] & M^{([n,\omega))}   \ar@{->}[d]^{\rho_{(n,m)}}\ar@{->}[r] &0 \\
\mathfrak{M}_m: & 0 \ar@{->}[r] &	M^{(m)} \ar@{->}[r]^{\nu_{<m}}  & M^{(\omega)}\ar@{->}[r]^{\rho_{\ge m}}  & M^{([m,\omega))} \ar@{->}[r] &0
}
\end{equation*}
commute,
\item the short exact sequence
\begin{equation*}
\xymatrix{ 
0 \ar@{->}[r] &	M^{(\omega)}  \ar@{->}[r]^{id} & M^{(\omega)} \ar@{->}[r]^{0} & 0 \ar@{->}[r] &0 
}
\end{equation*}
with morphisms $(\nu_{<n}, id, 0)$ forms a colimit of the
$\omega$-spectrum \\ $(\{\mathfrak{M}_n\}_n, \{(\nu_{(n,m)}, id, \rho_{(n,m)})\}_{n<m})$ in the category of complexes, 
\item $\rho_i\nu_{<n}=\rho_i$ if $i<n$ and $\rho_i\nu_{<n}=0$ otherwise.
\end{enumerate}
\end{lm}
\begin{proof}
An easy exercise of application of Lemma~\ref{StructMorph} in a Ab5-category.
\end{proof}

Before we formulate the categorial version of \cite[Proposition 1.1]{A-M} we prove a more general result:

\begin{lm}\label{notM-comp} Let for $M\in\Ac$ the category contain the products $M^\omega$.
The following conditions are equivalent for an object $N \in \Ac$:
\begin{enumerate}
\item $N$ is not $M$-compact,
\item there exists an $\omega$-spectrum $(\{N_i\}_{i<\omega}, \{\mu_{i,j}\}_{i<j<\omega})$ of $N$ with colimit  $(N,\{\mu_i\}_{i<\omega})$ such that all $\mu_i$ and $\mu_{i,j}$ for all $i<j<\omega$ are monomorphisms and for each $i<\omega$ there exists a nonzero morphism $\varphi_i\in\Ac(N,M)$ satisfying $\varphi_i\mu_i=0$,
\item there exists an $\omega$-spectrum with colimit  $(N,\{\mu_i\}_{i<\omega})$ such that for each $i<\omega$ there exists a nonzero morphism $\varphi_i\in\Ac(N,M)$ satisfying $\varphi_i\mu_i=0$.
\end{enumerate}
\end{lm}

\begin{proof} We will use the notation of Lemma~\ref{sums} throughout the whole proof.

 (1)$\Rightarrow$(2) Let $\varphi\in\Ac(N,M^{(\omega)})$
satisfying $\rho_i\varphi\ne 0$ for all $i<\omega$, which is ensured by (1) and  Theorem~\ref{N-comp}.
Furthermore, let us denote $\varphi_{\ge n}=\rho_{\ge n}\varphi $. Then 
$\rho_i\varphi=\rho_i\varphi_{\ge n}$ for all $i\ge n$. 
Now, for each $n<\omega$ denote by $(N_n,\mu_n)$ the kernel of the morphism $\varphi_{\ge n}$ and note that by the universal property of the kernel there exists a morphism $\mu_{n,n+1}\in \Ac(N_n, N_{n+1})$
such that all squares in the diagram with exact rows
\begin{equation*}
\xymatrix{ 
 0 \ar@{->}[r] &	N_n  \ar@{->}[r]^{\mu_n} \ar@{->}[d]_{\mu_{n,n+1}} & N \ar@{->}[r]^{\varphi_{\ge n}}\ar@{=}[d] & M^{([n,\omega))}   \ar@{->}[d]^{\rho_{(n,n+1)}} \\
 0 \ar@{->}[r] &	N_{n+1} \ar@{->}[r]^{\mu_{n+1}}  & N \ar@{->}[r]^{\varphi_{\ge n+1}}  & M^{([n+1,\omega))} 
}
\end{equation*}
commute. Now let us define inductively for each $n<m<\omega$ morphisms
\[
\mu_{n,m}:=\mu_{m-1,m}\mu_{m-2,m-1} \dots  \mu_{n+1,n+2}\mu_{n,n+1}\in \Ac(N_n, N_{m}).
\] 
If we denote by $(X,\{\xi_i\}_{i<\omega})$ the colimit of the $\omega$-spectrum $\Nc=(\{N_i\}_{i<\omega}, \{\mu_{i,j}\}_{i<j<\omega})$ then from Lemma~\ref{sums} we obtain the following commutative diagram with exact rows
\begin{equation*}
\xymatrix{ 
 0 \ar@{->}[r] &	N_n  \ar@{->}[r]^{\mu_n} \ar@{->}[d]_{\mu_{n,m}} & N \ar@{->}[r]^{\varphi_{\ge n}}\ar@{=}[d] & M^{([n,\omega))}   \ar@{->}[d]^{\rho_{(n,m)}} \\
 0 \ar@{->}[r] &	N_m \ar@{->}[r]^{\mu_{m}} \ar@{->}[d]_{\xi_{n}}  & N \ar@{->}[r]^{\varphi_{\ge m}} \ar@{=}[d] & M^{([m,\omega))} \ar@{->}[d]_{0}\\
 0 \ar@{->}[r] &	X \ar@{->}[r]^{\xi}  & N\ar@{->}[r]  & 0
}
\end{equation*}
because $\Ac$ is an Ab5-category. Thus $\xi$ is an isomorphism, which implies that $(N,\{\mu_i\}_{i<\omega})$ is a colimit  of the $\omega$-spectrum $\Nc$. Since all $\mu_n$'s are kernel morphisms, they are monomorphisms. Furthermore, $\mu_{n,m}$ are monomorphisms, because $\mu_{n}=\mu_{m}\mu_{n,m}$ for all $n<m<\omega$. Finally, put $\varphi_i:=\rho_i\varphi$, which is nonzero by the hypothesis, and compute
$\varphi_i\mu_i=\rho_i\varphi\mu_i=\rho_i\varphi_{\ge i}\mu_i=0$.

(2)$\Rightarrow$(3) This is clear.

(3)$\Rightarrow$(1) Let us denote by $\tau\in\Ac(N,M^\omega)$ the 
morphism satisfying $\pi_i\tau=\varphi_i$, which is (uniquely) given by the universal property of the product $M^\omega$. Recall that for each $n<\omega$ we have denoted by $\pi_{n}\in\Ac(M^\omega,M^n)$ the corresponding structural morphism  and we may identify objects $M^n$ and $M^{(n)}$ so we shall consider $\pi_{n}$ as a morphism in $\Ac(M^\omega,M^{(n)})$. 

Put $\tau_n:=\pi_{n}\tau\mu_n$. Since $\rho_i\tau_n=\rho_i\pi_{n}\tau\mu_n=\pi_i\pi_{n}\tau\mu_n$ we obtain that $\rho_i\tau_n=\varphi_i\mu_n\ne 0$ for each $i<n$ and  $\rho_i\tau_n=0$ for each $i\ge n$.
Then the diagram 
\begin{equation*}
\xymatrix{
	N_n \ar@{->}[r]^{\tau_n} \ar@{->}[d]_{\mu_{n,m}} & M^{(n)} 
	\ar@{->}[d]^{\nu_{(n,m)}} \\
	N_m  \ar@{->}[r]^{\tau_m}  & M^{(m)} } 
\end{equation*}
commutes for every $n<m<\omega$. Hence there exists $\varphi\in \Ac(N,M^{(\omega)})$ such that the diagram
\begin{equation*}
\xymatrix{
	N_n \ar@{->}[r]^{\tau_n} \ar@{->}[d]_{\mu_{n}} & M^{(n)} 
	\ar@{->}[d]^{\nu_{<n}} \\
	N  \ar@{->}[r]^{\varphi}  & M^{(\omega)} } 
\end{equation*}
commutes for each $n<\omega$ by Lemma~\ref{sums}, as $(N,\{\mu_i\}_{i<\omega})$ is the colimit  of the 
$\omega$-spectrum $(\{N_i\}_{i<\omega}, \{\mu_{i,j}\}_{i<j<\omega})$ and $(\{M^{(\omega)}\}_{i<\omega}, \{\mu_{(i)}\}_{i<\omega})$
is the colimit  of the 
$\omega$-spectrum $(\{M^{(i)}\}_{i<\omega}, \{\nu_{(i,j)}\}_{i<j<\omega})$ in the Ab5-category $\Ac$.

Applying Theorem~\ref{N-comp}, it is enough to prove that $\rho_i\varphi\ne 0$ for each $i<\omega$.
We have shown that $\rho_i\tau_n\ne 0$ for each $i<n$, hence
\[
\rho_i\varphi\mu_n = \rho_i\nu_{<n}\tau_n = \rho_i\tau_n\ne 0
\]
for each $i<n$, which implies $\rho_i\varphi\ne 0$.
\end{proof}

We are now ready to formulate a basic characterization of autocompact objects which generalizes the classical result \cite[Proposition 1.1]{A-M}. 

\begin{thm}\label{notAC}  Let $M$ be an object such that $\Ac$ is closed under  products $M^\lambda$ for all $\lambda\le\max(|\End_\Ac(M)|,\omega)$.
Then the following conditions are equivalent:
\begin{enumerate}
\item $M$ is not autocompact,
\item there exists an $\omega$-spectrum with colimit  $(M,\{\mu_i\}_{i<\omega})$ such that for each $i<\omega$ there exists a nonzero morphism $\varphi_i\in\End_\Ac(M)$ satisfying $\varphi_i\mu_i=0$,
\item there exists an $\omega$-spectrum $(\{M_i\}_{i<\omega}, \{\mu_{i,j}\}_{i<j<\omega})$ with colimit $(M,\{\mu_i\}_{i<\omega})$ such that $\{\Ic(M_i,\mu_i)\}_{i<\omega}$ forms a strictly decreasing chain of nonzero ideals of the ring $\End_\Ac(M)$ with $\bigcap_{i<\omega}\Ic(M_i,\mu_i)=0$.
\end{enumerate}
\end{thm}

\begin{proof} 
(1)$\Leftrightarrow$(2) follows form Lemma~\ref{notM-comp} for $N=M$.

(2)$\Rightarrow$(3) If $(M,\{\mu_i\}_{i<\omega})$ is the  colimit which exists by (2), then $\{\Ic(M_i,\mu_i)\}_{i<\omega}$ is  a decreasing chain of nonzero ideals of $\End_\Ac(M)$. Suppose that $\gamma\mu_i=0$ for all $i<\omega$. Then  $\gamma =0$, since there exists unique such morphism by the universal property of the colimit $(M,\{\nu_i\}_{i<\omega})$. Thus $\bigcap_{i<\omega}\Ic(M_i,\mu_i)=0$ and $\Ic(M_i,\mu_i)\ne 0$ for each $i$. If we put 
$J=\{j<\omega\mid \Ic(M_j,\mu_j)\ne\Ic(M_{j+1},\mu_{j+1})\}$, then it is easy to see that $(M,\{\mu_j\}_{j\in J})$ is the colimit of the $\omega$-spectrum  $(\{M_j\}_{j\in J}, \{\mu_{i,j}\}_{i<j\in J})$ with 
a strictly decreasing chain of nonzero ideals $\{\Ic(M_j,\mu_j)\}_{j\in J}$.

(3)$\Rightarrow$(2) It is enough to choose $\varphi_i\in \Ic(M_i,\mu_i)\setminus \Ic(M_{i+1},\mu_{i+1})$.
\end{proof}

The following criterion of autocompactness of finite coproducts generalizes results \cite[Proposition 5, Corollary 6]{Dvo2015} formulated in categories of modules.

\begin{prop}\label{finite_sums} The following conditions are equivalent for a finite family of  objects $\M$ and $M=\bigoplus \M$:
 \begin{enumerate}
\item $M$ is autocompact,
\item $N$ is $M$-compact for each $N\in \M$,
\item $M$ is $N$-compact for each $N\in \M$,
\item $N_1$ is $N_2$-compact for each $N_1,N_2\in \M$,
\item for each  $N_1,N_2\in \M$ and any $\omega$-spectrum $(\{K_i\}_{i<\omega}, \{\mu_{i,j}\}_{i<j<\omega})$ of $N_1$ with colimit  $(N_1,\{\mu_i\}_{i<\omega})$ and for each $i<\omega$ and nonzero $\varphi\in \Ac(N_1,N_2)$, the morphism $\varphi \mu_i$ is nonzero.
\end{enumerate}
\end{prop} 

\begin{proof} 
(1)$\Leftrightarrow$(4) This is proved in Proposition~\ref{sum-compact}  

(2)$\Leftrightarrow$(3)$\Leftrightarrow$(4) These equivalences follow from Proposition~\ref{sum-compact} again, when applied on pairs of families $\{M\},\M$ and $\M, \{M\}$.

(4)$\Leftrightarrow$(5) This is an immediate consequence of Lemma~\ref{notM-comp}.
\end{proof}	

As a consequence, we can formulate the assertion of Corollary~\ref{M^k} more precisely.

\begin{cor}\label{Oplus} Let $\M$ be a family of nonzero objects. Then $\bigoplus\M$ is autocompact if and only if  $\M$ is finite and $N_1$ is $N_2$-compact for each $N_1,N_2\in \M$. 
\end{cor}

The last direct consequence of Proposition~\ref{finite_sums} presents a categorial variant of  \cite[Corollary 7]{Dvo2015}.

\begin{cor}\label{fin} Let $\M$ be a finite family of autocompact objects
satisfying the condition $\Ac(N_1, N_2)=0$ whenever $N_1\ne N_2$. Then $\bigoplus\M$ is autocompact.
\end{cor}

If $\M$ is a finite family of objects, then $\bigoplus\M$ and $\prod\M$ are canonically isomorphic (cf. Lemma~\ref{StructMorph}), so the Proposition~\ref{finite_sums} holds true in case we replace any $\bigoplus$ by $\prod$ there. Although there is no autocompact coproduct of infinitely many nonzero objects by Corollary~\ref{sum-compact0}, the natural question that arises is, under which conditions the products of infinite families of objects are autocompact. The following example shows that the straightforward generalization of the claim does not hold true in general.

\begin{exm}\label{torsion} Denote by $\mathbb P$ the set of all prime numbers and consider the full subcategory $\mathcal{T}$ of the category of abelian groups $Ab$ consisting of all torsion abelian groups. If $A$ is a torsion abelian group and $A_{p}$ denotes its $p$-component for each $p\in\mathbb P$, then the decomposition $\bigoplus_{p\in\mathbb P}A_p$ forms both the coproduct and product of the family $\Ac=\{A_p\mid p\in\mathbb P\}$. 
Indeed, if $B$ is a torsion abelian group and $\tau_p\in Ab(B,A_p)$ for $p\in\mathbb P$, then for every $b\in B$ there exist only finitely many $p\in P$ for which $\tau_p(b)\ne 0$, hence the image of the homomorphism $f\in Ab(B,\prod_p A_p)$ given by the universal property of the product $\prod_p A_p$ is contained in $\bigoplus_{p\in\mathbb P}A_p$, hence $\bigoplus_{p\in\mathbb P}A_p$ is the  product of $\Ac$ in the category $\Tc$.

Thus, e.g. $\bigoplus_{p\in\mathbb P}\Z_p$ is the product of the family $\{\Z_p\mid p\in\mathbb P\}$ in $\Tc$, which is not autocompact in $\Tc$ by Corollary~\ref{Oplus}, however $\Z_p$ is $\Z_q$-compact for every $p,q\in\mathbb P$.
\end{exm}

\section{Which products are autocompact?}

Although the final section tries to answer the question formulated in its title, we start with one more closure property.

\begin{lm}\label{extensions}
If $0\to A\to B\to C\to 0$ is a short exact sequence such that an object $M$ is $A$-compact and $C$-compact, then it is $B$-compact.
\end{lm}
\begin{proof} Proving indirectly, assume that $M$ is not $B$-compact.
Then by Lemma~\ref{notM-comp} there exists a colimit $(M,\{\mu_i\}_{i<\omega})$ of some $\omega$-spectrum $(\{M_i\}_{i<\omega}, \{\mu_{i,j}\}_{i<j<\omega})$ and nonzero morphisms $\varphi_i\in\Ac(M,B)$ such that $\varphi_i\mu_i=0$, $i<\omega$.
If we suppose that $M$ is $C$-compact and  consider the
short exact sequence
\begin{equation*}
\xymatrix{ 
0 \ar@{->}[r] &	A  \ar@{->}[r]^{\alpha} & B \ar@{->}[r]^{\beta} & C \ar@{->}[r] &0,
}
\end{equation*}
then $\beta\varphi_i\mu_i=0$ for each $i\in\omega$, hence there exists $n$ such that $\beta\varphi_i=0$ for all $i\ge n$ by Lemma~\ref{notM-comp}.
By the universal property of  the kernel $\alpha$ of (the cokernel) $\beta$ there exist $\psi_i$ satisfying $\alpha\psi_i=\varphi_i\ne 0$ for each $i\ge n$. As $\alpha$ is a monomorphism, $\psi_i\ne 0$ for each $i\ge n$, hence $M$ is not $A$-compact by Lemma~\ref{notM-comp} again, a contradiction.
\end{proof}

\begin{cor} If $0\to A\to B\to C\to 0$ is a short exact sequence such that the object $B$ is $A$-compact and $C$-compact, then $B$ is autocompact.
\end{cor}

The previous corollary is a partial answer to the concluding question raised in \cite{Dvo2015}. As the next example shows, its assertion cannot be reversed.

\begin{exm} If we consider the  short exact sequence $0\to \Z\to \Q\to \Q/\Z\to 0$ in the category of abelian groups, then $\Q$ is self-small, i.e. autocompact abelian group and $\Z$-compact, but it is not $\Q/\Z$-compact.
\end{exm}

Now, we can formulate a criterion for autocompact objects which generalizes  \cite[Theorem 3.1]{DZ21}.

\begin{thm}\label{criterion}  Let $\M$ be a family of objects such that  the product $M=\prod\M$ exists in $\Ac$ and put $S=\bigoplus\M$.
Then the following conditions are equivalent:
\begin{enumerate}
\item $M$ is autocompact,
\item $M$ is $S$-compact,
\item $M$ is $\bigoplus\Cc$-compact for each countable family $\Cc\subseteq \M$.
\end{enumerate}
\end{thm}

\begin{proof} (1)$\Rightarrow$(2) Since $M$ is $\Add(M)$-compact by Lemma~\ref{Autocompact3} and $S=\bigoplus\M\in\Add(M)$, it is $S$-compact by Lemma~\ref{inclusion-comp}.

(2)$\Rightarrow$(3) This is an easy consequence of Proposition~\ref{sum-compact}.

(3)$\Rightarrow$(1) 
Assume on contrary that $M$ is not autocompact. Then by Lemma~\ref{notM-comp} there exist an $\omega$-spectrum $(\{M_i\}_{i<\omega}, \{\mu_{i,j}\}_{i<j<\omega})$ of $M$ with the colimit $(M,\{\mu_i\}_{i<\omega})$ such that $\mu_i$ is a monomorphism for all $i<\omega$ and for each $i<\omega$ there exists a nonzero morphism $\varphi_i\in \Ac(M,M)$ with $\varphi_i\mu_i=0$.
Then for each $i<\omega$ there exists $N_i\in \M$ such that $\pi_{N_i}\varphi_i\ne 0$.
Put $\Cc=\{N_i\mid i<\omega\}$ and denote by $\tilde{\nu}_{{N_i}}$ the structural morphisms of the coproduct $\bigoplus\Cc$. Since
$\tilde{\nu}_{N_i}\pi_{N_i}\varphi_i\in\Ac(M,\bigoplus\Cc)$ such that $\tilde{\nu}_{N_i}\pi_{N_i}\varphi_i\mu_i=0$, there exists $n$ for which  $\tilde{\nu}_{N_n}\pi_{N_n}\varphi_n=0$ by Lemma~\ref{notM-comp}, which contradicts the hypothesis $\pi_{N_i}\varphi_i\ne 0$ for each $i<\omega$.
\end{proof}

\begin{cor}\label{powerAC} Let  $M$ be an object and $I$ be a set. Then $M^I$ is autocompact if and only if $M^I$ is $M$-compact.
\end{cor}

Let us make a categorial observation about transfer of $\omega$-spectra via morphisms.

\begin{lm}\label{subspectrum} Let $G$ and $M$ be objects of $\Ac$ and $\alpha\in\Ac(G,M)$.
If $(\{M_i\}_{i<\omega}, \{\mu_{i,j}\}_{i<j<\omega})$ is an $\omega$-spectrum of $M$ with
the colimit $(M,\{\mu_i\}_{i<\omega})$ such that all $\mu_i$'s are monomorphisms,
\begin{enumerate}
\item then there exists an $\omega$-spectrum $(\{G_i\}_{i<\omega}, \{\gamma_{i,j}\}_{i<j<\omega})$ of $G$ with the colimit $(G,\{\gamma_i\}_{i<\omega})$ where $\gamma_i$ are monomorphisms for all $i$ and there exist morphisms $\alpha_i\in\Ac(G_i,M_i)$ such that the diagram
\begin{equation*}
\xymatrix{
	G_i \ar@{->}[r]^{\gamma_i} \ar@{->}[d]_{\alpha_i} & G
	\ar@{->}[d]^{\alpha} \\
	M_i  \ar@{->}[r]^{\mu_i}  & M } 
\end{equation*}
commutes for each $i<\omega$.
\item If $G$ is $A$-compact for an object $A$ and $t_i\in\Ac(M,A)$ are morphisms satisfying $t_i\mu_i=0$ for each $i<\omega$, then there exists $n$ such that $t_i\alpha=0$ for each $i\ge n$.
\end{enumerate}
\end{lm}

\begin{proof} (1) If we denote by $c_i\in\Ac(M,T_i)$ the cokernel of $\mu_i$ and $\gamma_i\in\Ac(G_i,G)$ the kernel of $c_i\alpha$ for every $i<j<\omega$, then $\mu_i$ is the kernel of $c_i$ and by the universal property of the kernel, there exists a morphism $\alpha_i\in\Ac(G_i,M_i)$ such that the diagram with exact rows 
\begin{equation*}
\xymatrix{
0 \ar@{->}[r] &	G_i \ar@{->}[r]^{\gamma_i} \ar@{->}[d]_{\alpha_i} & G \ar@{->}[r]^{c_i\alpha}\ar@{->}[d]^{\alpha} & T_i \ar@{=}[d]\\
0 \ar@{->}[r] &	M_i  \ar@{->}[r]^{\mu_i}  & M \ar@{->}[r]^{c_i} & T_i \ar@{->}[r]  &0 } 
\end{equation*}
commutes for each $i<\omega$. 
Furthermore, if we construct morphisms $\gamma_{i,j}$, $i<j<\omega$  using the universal property of the kernels as in the proof of
Lemma~\ref{notM-comp} (1)$\Rightarrow$(2), then we get the following commutative diagram
\begin{equation*}
\xymatrix{
	G_i \ar@{->}[r]^{\gamma_i} \ar@{->}[d]_{\gamma_{i,j}} & G
	\ar@{=}[d] \\
	G_j  \ar@{->}[r]^{\gamma_{j}}  & G }
\end{equation*}
and checking that $(G,\{\gamma_i\}_{i<\omega})$ is a colimit of the $\omega$-spectrum $(\{G_i\}_{i<\omega}, \{\gamma_{i,j}\}_{i<j<\omega})$ of $G$ is easy.

(2) From (1) we have an $\omega$-spectrum $(\{G_i\}_{i<\omega}, \{\gamma_{i,j}\}_{i<j<\omega})$  with the colimit $(G,\{\gamma_i\}_{i<\omega})$ and morphisms 
$\alpha_i\in\Ac(G_i,M_i)$ such that the diagram 
\begin{equation*}
\xymatrix{
0\ar@{->}[r] &	G_i \ar@{->}[r]^{\gamma_i} \ar@{->}[d]_{\alpha_i} & G \ar@{->}[d]^{\alpha} 
	\\
0\ar@{->}[r]  &	M_i  \ar@{->}[r]^{\mu_i}  & M \ar@{->}[r]^{t_i} & A } 
\end{equation*}
commutes for every $i<\omega$. Since $t_i\alpha\gamma_i=t_i\mu_i\alpha_i=0$ by the hypothesis, Lemma~\ref{notM-comp} applied on the $A$-compact object $G$ and morphisms $t_i\alpha$, $i<\omega$, give us $n$ such that $t_i\alpha=0$ for all $i\ge n$.
\end{proof}

\begin{lm}\label{(A,B)}
Let $A$ and $B$ be objects of $\Ac$ and $\Ac(A,B)=0$.
If $\alpha\in \Ac(A\prod B,B)$ then there  exists $\tau\in\Ac(B,B)$ for which $\alpha=\tau\pi_B$.
\end{lm}

\begin{proof} Since $0\to A \stackrel{\nu_A}{\longrightarrow}  A\oplus B \stackrel{\rho_B}{\longrightarrow} B\to 0$ is  a  short exact sequence and $\alpha\nu_A=0$, the claim follows from the universal property of the cokernel $\rho_B$ and by applying the canonical isomorphism $A\oplus B\cong A\prod B$.
\end{proof}

Recall that $G$ is a \emph{projective generator} of $\Ac$, if for any nonzero object $B$ in $\Ac$, $\Ac(G,B)\ne 0$ holds and for each pair of objects $A$, $B$, any epimorphism $\pi\in \Ac(A,B)$ and any morphism $\varphi\in \Ac(G,B)$ there exists $\tau\in \Ac(G,A)$ such that $\varphi=\pi\tau$.

The following assertion is a categorial version of \cite[Proposition 1.6]{Z08} (cf. also \cite[Corollary 1.3]{A-M}). Call an $\Ac$-compact object briefly \emph{compact} object.

\begin{thm}\label{prod} Let $\M$ be a family of objects,  $\Ac$ contain a compact projective generator and the product $M = \prod \M$. Denote $M_N=\prod(\M\setminus\{N\})$ and let $\Ac(M_N, N)=0$ for each $N$. Then $M$ is autocompact if and only if $N$ is autocompact for each $N\in\M$. 
\end{thm}

\begin{proof} ($\Rightarrow$) Since $M\cong N\oplus M_N$ for every $N\in\M$, the assertion follows from Proposition~\ref{finite_sums}.

($\Leftarrow$) First note that $M_N$ is a trivial example of an $N$-compact module (cf. Example~\ref{ex1}), so $M$ is $N$-compact for every $N\in\M$ by Proposition~\ref{finite_sums}.

Assume that $M$ is not $M$-compact, hence by Lemma~\ref{notM-comp} there exists an $\omega$-spectrum $(\{M_i\}_{i<\omega}, \{\mu_{i,j}\}_{i<j<\omega})$  with the colimit $(M,\{\mu_i\}_{i<\omega})$ such that for each $i<\omega$  there exists a nonzero $\tilde{\varphi}_i\in \Ac(M,M)$ and $N_i\in\M$ for which $\tilde{\varphi}_i\mu_i=0$ and $\pi_{N_i}\tilde{\varphi}_i\ne 0$. Put $\varphi_i=\pi_{N_i}\tilde{\varphi}_i$ for each $i<\omega$ and $\Cc=\{N_i\mid i<\omega\}$ and note that there exist $\psi_i\in\Ac(N_i,N_i)$ satisfying $\psi_i\pi_{N_i}=\varphi_i$ by Lemma~\ref{(A,B)} applied on $M_{N_i}\prod N_i$ for each $i<\omega$.

If $\Cc$ is finite, then $M$ is $\prod\Cc$-compact by Proposition~\ref{sum-compact} applied on $\{M\}$ and $\Cc$, hence there exists $n$ such that $\varphi_n=\pi_{N_n}\pi_\Cc\varphi=\pi_{N_n}\varphi=0$, which contradicts the fact that $\varphi_i\ne 0$ for all $i<\omega$.

Thus $\Cc$ is infinite and we may assume w.l.o.g. that $N_i\ne N_j$ whenever $i\ne j$. Denote the cokernel of the composition $\pi_{N_i}\mu_i$ by $\sigma_i\in\Ac(N_i,T_i)$ for $i<\omega$. Then we have a commutative diagram
\begin{equation*}
\xymatrix{
0\ar@{->}[r] & M_i\ar@{->}[r]^{\mu_i} &	M \ar@{->}[r]^{\pi_{N_i}} \ar@{->}[d]_{\varphi_i} & N_i \ar@{->}[r]^{\sigma_i}\ar@{->}[d]^{\psi_i} & T_i \ar@{=}[r] \ar@{->}[d]^{\tilde{\mu}_{T_i}} & T_i \ar@{=}[d] \ar@{->}[r] & 0
	\\
&  &	N_i  \ar@{=}[r]  & N_i & \prod_j T_j \ar@{->}[r]^{\tilde{\pi}_{T_i}}  & T_i \ar@{->}[r] & 0 } 
\end{equation*}
 for each $i<\omega$, where $\tilde{\pi}_{T_i}$ and $\tilde{\mu}_{T_i}$ denote the structural and associated morphisms of the product $\prod_j T_j$. Since $\psi_i\pi_{N_i}\mu_i=\varphi_i\mu_i=0$ and $\psi_i\ne 0$, the morphism $\pi_{N_i}\mu_i$ is not an epimorphism and so $T_i\ne 0$. As $G$ is a  projective generator, there exists $\zeta_i\in \Ac(G,N_i)$ satisfying $\sigma_i\zeta_i\ne 0$ for each $i<\omega$. Then by the universal property of the product $\prod\Cc$, there is $\zeta\in \Ac(G,\prod\Cc)$ such that $\hat{\pi}_{N_i}\zeta=\zeta_i$, hence $\sigma_i\hat{\pi}_{N_i}\zeta=\sigma_i\zeta_i\ne 0$ for all $i<\omega$. 
If we define $t_i=\tilde{\mu}_{T_i}\sigma_i\pi_{N_i}$ and denote by $\mu_{\Cc}\in\Ac(\prod \Cc, M)$ the associated morphism, we can easily compute 
\[
t_i\mu_i=\tilde{\mu}_{T_i}\sigma_i\pi_{N_i}\mu_i=0\ \ \text{and}\ \ 
t_i\mu_{\Cc}\zeta=\tilde{\mu}_{T_i}\sigma_i\pi_{N_i}\mu_{\Cc}\zeta=\tilde{\mu}_{T_i}\sigma_i\zeta_i\ne 0
\]
as $\tilde{\mu}_{T_i}$ for $i<\omega$ is a monomorphism, which contradicts the hypothesis that $G$ is compact by Lemma~\ref{subspectrum}(2).
\end{proof}

The following example shows that the existence of the compact projective  generator cannot be removed from the assumptions of the last assertion.

\begin{exm}  Consider the category  of all torsion abelian groups $\mathcal{T}$ from Example~\ref{torsion}. Then $M=\bigoplus_{q\in\mathbb P}\Z_q$ is the product of the family $\{\Z_q\mid q\in\mathbb P\}$ and $M_p=\bigoplus_{q\ne p}\Z_p$ is the product of the family $\{\Z_q\mid q\in\mathbb P\setminus\{p\}\}$ for all $p\in\mathbb P$
in the category $\Tc$. Although $\Hom_\Tc(M_p,\Z_p)=0$ and $\Z_p$ is autocompact in $\Tc$ for each $p\in\mathbb P$, $M$ is not autocompact.
Let us remark that the category $\Tc$ contains no compact generator.
\cite[Corollary 1.8]{Z08}).
\end{exm}

We conclude with a well-known example of an autocompact product.

\begin{exm} Any finitely generated free abelian group is a compact projective  generator in the category of abelian groups and the family $\{\Z_q\mid q\in\mathbb P\}$ satisfies the hypothesis of Proposition~\ref{prod} by \cite[Lemma 1.7]{Z08}, hence $\prod_{q\in\mathbb P}\Z_q$ is autocompact (cf.\cite[Corollary 1.8]{Z08}).
\end{exm}

\def\bibname{Bibliography}

\end{document}